\documentclass[a4paper,11pt]{article}
\usepackage{latexsym}
\usepackage{amsfonts,amssymb,amsthm,amsmath}
\usepackage{amscd}
\usepackage[dvips]{graphicx}
\usepackage{color}
\usepackage{fullpage} 
\usepackage[width=.85\textwidth]{caption}
\usepackage[mathlines]{lineno}

\author{ Maciej Dziemia\'nczuk\footnote{
Institute of Informatics, University of Gda\'nsk, Poland;  E-mail: \texttt{mdziemia@inf.ug.edu.pl}.}
} 
\title{On Directed Lattice Paths With Additional Vertical Steps}

\newtheorem{theorem}{Theorem}

\newtheorem{proposition}{Proposition}

\newtheorem{corollary}{Corollary}
\newtheorem{lemma}{Lemma}

\theoremstyle{definition}
\newtheorem{example}{Example}

\begin{document}
\maketitle

\begin{abstract}
The paper is devoted to the study of lattice paths that consist of vertical steps $(0,-1)$ and non-vertical steps $(1,k)$ for some $k\in \mathbb Z$.
Two special families of primary and free lattice paths with vertical steps are considered.
It is shown that for any family of primary paths there are equinumerous families of proper weighted lattice paths that consist of only non-vertical steps.
The relation between primary and free paths is established and some combinatorial and statistical properties are obtained. It is shown that the expected number of vertical steps in a primary path running from $(0,0)$ to $(n,-1)$ is equal to the number of free paths running from $(0,0)$ to $(n,0)$. Enumerative results with generating functions are given.
Finally, a few examples of families of paths with vertical steps are presented and related to {\L}ukasiewicz, Motzkin, Dyck and Delannoy paths.
\end{abstract}

\section{Introduction}

A \emph{lattice path} is a sequence of points from $\mathbb Z\times \mathbb Z$. A pair $(i,j)$ from $\mathbb Z\times \mathbb Z$ is called a \emph{step} if $(i,j)\neq (0,0)$. An example of a lattice path is given in Fig.~\ref{fig:examp1}. For simplicity of notation, we represent lattice paths as the words over fixed set of steps $\mathcal S$. We shall say that a path is an $\mathcal S$-path if its steps belong to $\mathcal S$. Lattice paths appear in many contexts. They are used in physics \cite{Rensburg}, computer science \cite{KnuthVol2}, and probability theory \cite{TakasBook}. 
There are a huge number of papers on enumeration of lattice paths for specified sets of steps. We refer the reader to the survey of Humphreys \cite{Humphreys} and to the references therein.

\begin{figure}[ht]
\begin{center}
	\includegraphics[width=30mm]{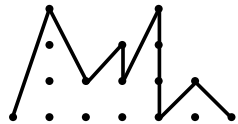}
	\caption{A lattice path running from $(0,0)$ to $(6,0)$.}
	\label{fig:examp1}
\end{center}
\end{figure}

In this paper we consider only the following types of steps. Namely, 
let $V$ denote the \emph{vertical step} $(0,-1)$ and let $S_k$ denote the \emph{non-vertical step} $(1,k)$, for $k\in\mathbb Z$. Additionally, we separate non-vertical steps into two groups. If $k\geq 0$ then $S_k$ is called \emph{up step} and denoted by $U_k$. If $k<0$ then $S_k$ is called \emph{down step} and denoted by $D_{-k}$. 
For example, the path in Fig.~\ref{fig:examp1} can be represented by its starting point $(0,0)$ and the sequence $U_3 D_2 U_1 V U_2 V^3 U_1 D_1$.

There are several well-known examples of paths that consist of non-vertical steps. For instance, Dyck paths are composed of steps $U_1$ and $D_1$, Motzkin and $N$-{\L}ukasiewicz paths are those for which the sets of steps are $\{U_1, U_0, D_1\}$ and $\{ U_N, U_{N-1}, \ldots , U_0, D_1\}$, respectively. All of these examples are essentially one-dimensional objects. It is well known, see e.g. Deutsch \cite{Deutsch2}, that the number of Dyck paths running from $(0,0)$ to $(2n,n)$ which never go below the $x$-axis is the $n$th Catalan number given by
\[
	C_n = \frac{1}{n+1} \binom{2n}{n}.
\]
The number of Motzkin paths \cite{Shapiro1} running from $(0,0)$ to $(n,0)$ which never go below the $x$-axis is equal to
\[
	\sum_{k=0}^{\lfloor \frac{n}{2} \rfloor} \binom{n}{2k} C_k.
\]
A unified enumerative and asymptotic theory of paths consisting of non-vertical steps is developed by Banderier and Flajolet \cite{Banderier2002} and concerns with the kernel method.

\begin{figure}[ht]
\begin{center}
	\includegraphics[width=57mm]{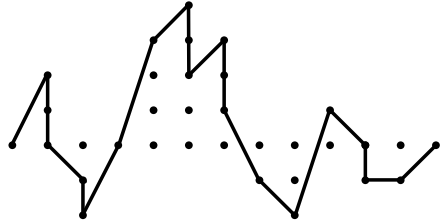}
	\caption{A free $\mathcal S$-path running from $(0,0)$ to $(12,0)$.}
	\label{fig:2}
\end{center}
\end{figure}

Another class of lattice paths are those which consist of vertical and non-vertical steps. The classical example is the family of Delannoy paths \cite{Delannoy} which originally consist of steps $(1,0)$, $(1,1)$, and $(0,1)$. In our notation, a \emph{Delannoy path} is a lattice path that consists of steps $U_0, D_1, V$ and runs from $(0,0)$ to some $(n,m)$ in the fourth quarter, i.e., $n\geq 0$ and $m\leq 0$. The number of Delannoy paths running from $(0,0)$ to $(n,m)$ is equal to 
\[
	\sum_{j=0}^n \binom{m}{j} \binom{n+m-j}{m}.
\]
Several families of paths with steps from $\{V, U_1, U_0, D_1 \}$ are considered by the author in \cite{mdRect}.

\vspace{0.2cm}
In this paper we consider families of paths with the set of steps being an arbitrary subset of $\{ V, S_N, S_{N-1},\ldots \}$, for fixed $N\geq 0$. 
Throughout the paper we use the symbol $\Omega_N$ to denote the following set of steps
\[
	\Omega_N = \{ S_k : k\leq N \},
	\qquad
	\text{for} \,\, N\geq 0.
\]
For each $m,n\in\mathbb Z$ and $\mathcal S \subset \{V\} \cup \Omega_N$, we define two families of paths.
Namely, let $\mathcal F_{\mathcal S}(m,n)$ denote the set of all $\mathcal S$-path running from $(0,0)$ to $(n,-m)$. Let $\mathcal P_{\mathcal S}(m,n)$ denote the set of these paths from $\mathcal F_{\mathcal S}(m,n)$ for which all points except possibly the last one lie on or above the $x$-axis.
We call a path from $\mathcal F_{\mathcal S}(m,n)$ a \emph{free path} and a path from $\mathcal P_{\mathcal S}(m,n)$ an \emph{$m$-primary path}.
For instance, a free path is given in Fig.~\ref{fig:2} and a primary path is given in Fig.~\ref{fig:1}.

\begin{figure}[ht]
\begin{center}
	\includegraphics[width=35mm]{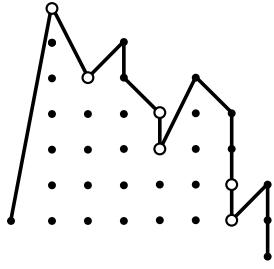}
	\caption{A $1$-primary $\mathcal S$-path running from $(0,0)$ to $(7,-1)$. Lattice points which determine a decomposition of the path are drawn by open circles.}
	\label{fig:1}
\end{center}
\end{figure}

In Section~\ref{sec:biject} we show that for any $\mathcal V \subseteq \Omega_N \cup \{V\}$ which contains $U_N$ and $V$ there is the corresponding set of steps 
\[
	\mathcal L = \big( \mathcal V \setminus \{V\} \big) \cup \{ D_1, U_0, U_1, \ldots, U_N \}
\] such that for any $m\geq 0$ and $n\geq 1$, we have
\begin{equation}
	\label{eq:biject}
	|\mathcal P_{\mathcal V}(m,n)| = \sum_{\pi\in\mathcal P_{\mathcal L}(m,n)} w(\pi),
\end{equation}
where $w$ is a weight function over paths from $\mathcal P_{\mathcal L}(m,n)$.
This means that additional vertical step $V$ in primary $\mathcal V$-paths can be encoded by the proper weights of non-vertical steps in $\mathcal L$-paths. 
To show the above equality we define $w$-weighted primary $\mathcal L$-paths which are primary $\mathcal L$-paths whose steps have assigned nonnegative integers depending on the weight function $w$. Then we define a bijection between primary $\mathcal V$-paths and $w$-weighted primary $\mathcal L$-paths.

In Section~\ref{sec:properties} we establish a relation between primary and free paths. Namely, we show that for any $\mathcal V \subset \Omega_N\cup\{V\}$ which consists $V$ and $U_N$, we have
\begin{align*}
	|\mathcal P_{\mathcal V}(1,n)| 
	&= \frac{1}{n} 
	\big( 
	|\mathcal F_{\mathcal V}(1,n)| - |\mathcal F_{\mathcal V}(0,n)| 
	\big)
	\\
	&= \frac{1}{n} \sum_{j=0}^{Nn+1} \binom{n+j-1}{j}
	|\mathcal F_{\mathcal N}(1-j,n)|,
\end{align*} 
where $\mathcal N = \mathcal V \setminus \{V\}$.
We show that 
\begin{align*}
	\#Steps(V \in \mathcal P_{\mathcal V}(1,n))
	&=
	|\mathcal F_{\mathcal V}(0,n)|,
	\\
	\#Steps(S_k \in \mathcal P_{\mathcal V}(1,n))
	&=
	|\mathcal F_{\mathcal V}(1+k,n-1)|,
	\\
	\#Steps(\mathcal P_{\mathcal V}(1,n))
	&= |\mathcal F_{\mathcal V}(1,n)|,
\end{align*}
where $\#Steps(S \in \mathcal P_{\mathcal V}(1,n))$ denote the total number of occurrences of the step $S$ in the set of paths from $\mathcal P_{\mathcal V}(1,n)$
and $\#Steps(\mathcal P_{\mathcal V}(1,n))$ denote the total number of all steps in $\mathcal P_{\mathcal V}(1,n)$.
These relations shed some light on the statistical properties of the $\mathcal V$-paths. Any $\mathcal N$-path running from $(0,0)$ to $(n,m)$ contains exactly $n$ steps. The number of steps in a $\mathcal V$-path is equal to or greater than $n$. We show that the expected number of steps in a $1$-primary $\mathcal V$-path running from $(0,0)$ to $(n,-1)$ is equal to
\[
	n \left( 1 + \frac{ |\mathcal F_{\mathcal V}(0,n)|}{|\mathcal F_{\mathcal V}(1,n)| - |\mathcal F_{\mathcal V}(0,n)| } \right).
\]

In Section~\ref{sec:enumer} we derive some enumerative results for such paths. We show that, for any $n\geq 1$ and $m\in\mathbb Z$, we have
\begin{align*}
	|\mathcal P_{\mathcal V}(1,n)| 
	&= \frac{1}{n} [y^{Nn+1}] \frac{1}{(1-y)^n} 
	\Big( \sum_{S_k\in\mathcal V} y^{N-k} \Big)^n,
	\\
	|\mathcal F_{\mathcal V}(m,n) | 
	&= [y^{N n + m}] \frac{1}{(1-y)^{n+1}} 
	\Big( \sum_{S_k\in\mathcal N} y^{N-k} \Big)^n.
\end{align*}
In Section~\ref{sec:riordan} we show that the array $(|\mathcal F_{\mathcal V}(i-(N+1)j,j)|)_{i,j\geq 0}$ is the proper Riordan array $D_{\mathcal V}$, where
\[
	D_{\mathcal V} = \left( \frac{1}{1-y}, \frac{y}{1-y} \sum_{S_k\in\mathcal V} y^{N-k}  \right).
\]
Let $P_{m}(x)$ be the generating function of the sequence $(|\mathcal P_{\mathcal V}(m,n)|)_{n\geq 0}$. As a consequence of \eqref{eq:biject}, we show  that $P_{m}(x)$ satisfies the following functional equation
\begin{align*}
	P_{0}(x) 
	&= 1 + w(U_0^{0,0}) x P_0(x)
	+ x P_{0}(x) \sum_{k=1}^N
	\sum_{d=1}^{k} 
	w(U_k^{0,d})
	\sum_{M}
	\prod_{j=1}^d 
	(P_{m_j}(x)-1),
	\\
	P_{m}(x) 
	&= 1 + w(D_m) x + x \sum_{k=0}^N \sum_{d=1}^{k+1}  
	w(U_{k}^{m,d})  
	\sum_{M}
	\prod_{j=1}^d (P_{m_j}(x)-1),
\end{align*}
for certain constants $w(D_m)$ and $w(U^{m,d}_k)$.

In Section~\ref{sec:examples} we present five examples of families of lattice paths with vertical steps for which we apply results obtained in the previous sections. Namely, we consider the following five sets of lattice steps:
\begin{align*}
	\mathcal A &= \{V\} \cup \{ S_k : K\leq k\leq N \}
	= \{V, U_N, U_{N-1} ,\ldots, U_0, D_1, \ldots, D_K \},
	\\
	\mathcal B &= \{V\} \cup \{ S_k : 1\leq k\leq N \}
	= \{V, U_N, U_{N-1} ,\ldots, U_0, D_1, \},
	\\
	\mathcal C &= \{V\} \cup \Omega_N
	= \{V, U_N, U_{N-1} ,\ldots, U_0, D_1, D_2, \ldots \},
	\\
	\mathcal D &= \{V, U_N, D_K \},
	\\
	\mathcal E &= \{V, U_1, U_0 \}.
\end{align*}
for any fixed $N\geq 0$ and $K\geq 1$. It is worth pointing out that $\mathcal B$-paths, $\mathcal C$-paths, $\mathcal D$-paths, and $\mathcal E$-path are generalized {\L}ukasiewicz paths, Raney paths recently considered by the author in \cite{mdRaney}, generalized Dyck paths, and Delannoy paths, respectively.

\section{Preliminaries}
\label{sec:pre}

Let $N\geq 0$ be a nonnegative integer and let $\mathcal S$ be a subset of $\Omega_N\cup\{V\}$ and $m,n\in\mathbb Z$.
Any $m$-primary $\mathcal S$-path $\mu\in\mathcal P_{\mathcal S}(m,n)$ which has at least two non-vertical steps can be uniquely decomposed into some number of vertical steps and shorter primary $\mathcal S$-paths. Let $U_h$ be the first step of $\mu$. First, suppose that $(m,h) \neq (0,0)$.
The path $\mu$ passes through the points $(x_1,h), (x_2, h-1),\ldots,(x_{h+t},-m+1)\in\mathbb R\times \mathbb Z$, such that they are chosen to be the left-most ones, i.e.,
$x_i = \min\{ x : \mu \text{ passes through } (x,h-i+1) \}$.
Note that $x_1 = 1$ and some of $x_i$ may be not integer. Therefore, denote by $\Pi(\mu)$ the set of these points that both coordinates are integers. For instance, points from $\Pi$ for the path given in Fig.~\ref{fig:1} are marked by open circles. Cutting $\mu$ at these points we obtain a decomposition of $\mu$ into $U_h$ and $r$ subpaths $\alpha^{(1)},\ldots,\alpha^{(r)}$, where $r=|\Pi(\mu)|$. Each $\alpha^{(i)}$ is either a single vertical step $V$ or an $m_i$-primary $\mathcal S$-path for some $m_i\geq 1$. Suppose that exactly $d$ of $\alpha^{(1)},\ldots,\alpha^{(r)}$ are not vertical steps and denote them by $\mu^{(1)},\ldots,\mu^{(d)}$. Hence, $\mu$ can be uniquely decomposed as
\[
	\mu = U_h V^{\lambda_0} \, \mu^{(1)} V^{\lambda_1} \mu^{(2)} V^{\lambda_2} \cdots \mu^{(d)} V^{\lambda_d},
\]
where $\mu^{(i)}$ is an $m_i$-primary $\mathcal S$-path and $\lambda_0,\lambda_1,\ldots,\lambda_d$ are nonnegative integers. It is worth pointing out that $m_1,\ldots,m_{d-1}\geq 1$, $m_d\geq m$, and $\lambda_d = 0$ if $m\geq 2$.
The \emph{shape} of the path $\mu$ is the triple $(m,d,k)$, where $k = h - \lambda_0-\lambda_1-\cdots-\lambda_d$.
Observe that any up step in $\mu$ is the first step of the uniquely determined  primary $\mathcal S$-subpath of $\mu$. For this reason, we denote by $U_{k}^{m,d}$ the up step $U_k$ which is the first step of an $m$-primary $\mathcal S$-path whose shape is $(m,d,k)$.

Finally, if $(m,h) = (0,0)$ then the path $\mu$ is decomposable into $U_0$ and $\mu^{(0)}$, where $\mu^{(0)}$ is a $0$-primary $\mathcal S$-path from the set $\mathcal P_{\mathcal S}(0,n-1)$. To simplify the further consideration, we assume that the shape of $\mu$ in this case is $(0,0,0)$ and the first step $U_0$ of such path is denoted by $U_0^{0,0}$.

\begin{figure}[ht]
\begin{center}
	\includegraphics[width=35mm]{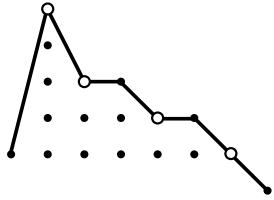}
	\caption{A $1$-primary $\mathcal S$-path running from $(0,0)$ to $(7,-1)$ without vertical steps. Lattice points determining a decomposition of the path are drawn by open circles.}
	\label{fig:3}
\end{center}
\end{figure}

For instance, if the set of steps $\mathcal S$ does not contain vertical step $V$, then 
any $m$-primary $\mathcal S$-path $\mu$ running from $(0,0)$ to $(n,-m)$ can be uniquely decomposed as
\[
	\mu = U_k \, \mu^{(1)} \mu^{(2)} \cdots \mu^{(d)},
\]
where $d = |\Pi(\mu)|$. See Fig.~\ref{fig:3}.

\begin{example}
\label{exam:decompos}
Let  $\pi = U_6 D_2 U_1 V D_1 V U_2 D_1 V^3 U_1 V^2$ be the path from Fig.~\ref{fig:1}. This path is decomposable as $U_6^{m,d} \pi^{(1)} \pi^{(2)} V \pi^{(3)} V \pi^{(4)}$, where $m=1$, $d=4$, and
\[
	\pi^{(1)} = D_2,
	\quad
	\pi^{(2)} = U_1 V D_1,
	\quad
	\pi^{(3)} = U_2 D_1 V^2,
	\quad
	\pi^{(4)} = U_1 V^2.
\]
The shape of $\pi$ is $(1,4,4)$ and $\lambda_0 = \lambda_1 = \lambda_4 = 0$, $\lambda_2=\lambda_3 = 1$.
\end{example}

\vspace{0.2cm}
For $k\in\mathbb Z$, let $\mathcal V_{\geq k}$ denote the set of steps $\mathcal V \cap \{S_h : h\geq k \}$. For $m,k,d\geq 0$, let $\mathcal H_{\mathcal V}(m,d,k)$ denote the set of all pairs $(h,\lambda)$, where $h$ is an integer such that $S_h \in \mathcal V_{\geq k}$ and $\lambda$ is a composition of $h-k$ into $d+1$ parts if $m\in\{0,1\}$, or into $d$ parts if $m\geq 2$ (zero parts are allowed in both cases).

\begin{proposition}
\label{prop:h}
For any $m,d,k\geq 0$, we have
\begin{equation}
	|\mathcal H_{\mathcal V}(m,d,k)| = \sum_{U_h\in \mathcal V_{\geq k} }
	\binom{h-k+d-\epsilon_m}{h-k},
\end{equation}
where $\epsilon_m = 0$ if $m\in\{0,1\}$, and $\epsilon_m = 1$ if $m\geq 2$.
\end{proposition}
\begin{proof}
Recall that $N$ is the maximal integer such that $U_N\in\mathcal V$. Let us partition the set $\mathcal H_{\mathcal V}(m,d,k)$ into pairwise disjoint classes $A_k, A_{k+1}, \ldots, A_N$, where $A_h$ contains these pairs whose first element is $h$. If $U_h\notin \mathcal V_{\geq k}$ then $A_h$ is empty. If $U_h\in \mathcal V_{\geq k}$ then the size of $A_h$ is the number of compositions of $h-k$ into $d+1$ possibly zero parts, for $m\in\{0,1\}$, or into $d$ possibly zero parts, for $m\geq 2$, which is equal to the value of the binomial coefficient in the final formula.
\end{proof}

\section{Bijection between primary paths}
\label{sec:biject}

Let $N\geq 0$ and $\mathcal V$ be a subset of the set of steps $\Omega_N\cup\{V\}$ such that $U_N\in\mathcal V$ and $V\in\mathcal V$. Now we define the corresponding set of steps $\mathcal L$ which does not contain $V$. Namely, let
\begin{equation}
	\label{eq:defN}\mathcal L = ( \mathcal V \setminus \{V\} )\, \cup \, \{ U_N, U_{N-1}, \ldots, U_0, D_1\}.
\end{equation}

Let us define weighted paths. First, we define a weight function $w_{\mathcal V}$ over steps in primary $\mathcal L$-paths as follows.
Let $D_p$ be a down step from $\mathcal L$ and $U_{k}^{m,d}$ be an up step $U_k$ which is the first step of an $m$-primary $\mathcal L$-path whose shape is $(m,d,k)$. Then we set
\begin{align}
	\label{eq:weight}
	w_{\mathcal V}(D_p) = \left\{ 
\begin{array}{cl}
	|\mathcal V_{\geq -1}| & \mathrm{if}\,\, p = 1, \\ 
	1 & \mathrm{if}\,\, p \geq 2,
\end{array}		
	\right.
	\qquad
	w_{\mathcal V}(U_{k}^{m,d}) 
	= |\mathcal H_{\mathcal V}(m,d,k)|,
\end{align}
for all $m,d,k\geq 0$. We write $w(S)$ instead of $w_{\mathcal V}(S)$ for short if no confusion can arise. A \emph{weighted $m$-primary $\mathcal L$-path} is a pair $(\mu,v)$, where $\mu$ is an $m$-primary $\mathcal L$-path consisting of $n$ steps $\mu_1,\ldots,\mu_n$ and $v$ is a sequence of $n$ positive integers $v_1,\ldots,v_n$, called \emph{weights}, such that $1\leq v_i\leq w(\mu_i)$.
A \emph{weight} of a $\mu$, denoted by $w_{\mathcal V}(\mu)$, is a product of weights of its steps.
Let $\mathcal W^{\mathcal V}_{\mathcal L}(m,n)$ denote the set of all $w_{\mathcal V}$-weighted $m$-primary $\mathcal L$-paths running from $(0,0)$ to $(n,-m)$. It is clear that we have
\[
	|\mathcal W^{\mathcal V}_{\mathcal L}(m,n)| = \sum_{\mu \in \mathcal P_{\mathcal L}(m,n)} w_{\mathcal V}(\mu).
\]

\begin{theorem}
\label{th:biject}
For any $n\geq 1$ and $m\geq 0$, we have
\begin{equation}
	|\mathcal P_\mathcal V(m,n)| = |\mathcal W_{\mathcal L}(m,n)|.
\end{equation}
\end{theorem}
\begin{proof}
In the sequel, we define a map $f_{m,n}:\mathcal W_{\mathcal L}(m,n) \to \mathcal P_{\mathcal V}(m,n)$, for any $m\geq 0$ and $n\geq 1$. In Lemma~\ref{lem:biject}, given below, we show that this map is a bijection.
\end{proof}

\vspace{0.2cm}
\textbf{The map} $f_{m,n}:\mathcal W_{\mathcal L}(m,n)  \to \mathcal P_{\mathcal V}(m,n)$, for $m\geq 0$, $n\geq 1$.

\vspace{0.2cm}
Let $(\mu, v)$ be a weighted path from $\mathcal W_{\mathcal L}(m,n)$ and $v=(v_1,\ldots,v_n)$. If $n=1$ and $m=0$, then $\mu=U_0$, and $1\leq v_1 \leq |\mathcal H_{\mathcal V}(0,0,0)|$. Suppose that $(h,\lambda)$ is the $v_1$th pair from $\mathcal H_{\mathcal V}(0,0,0)$. Note that $\lambda = (h)$. Then we set 
\[
	f_{0,1}((U_0,v))
	\overset{\mathrm{def}}{=\joinrel=} 
	U_h V^{h}.
\]
If $n=1$ and $m=1$, then $\mu=D_1$, and $1\leq v_1 \leq |\mathcal V_{\geq -1}|$.
Suppose that $S_h$ is the $v_1$th step from $\mathcal V_{\geq -1}$. We set
\[
	f_{1,1}((D_1,v))
	\overset{\mathrm{def}}{=\joinrel=} 
	S_h V^{h+1}.
\]
If $n=1$ and $m\geq 2$, then $\mu = D_m$ and $v_1 = 1$. We set
\[
	f_{m,1}((D_m,v)) 
	\overset{\mathrm{def}}{=\joinrel=} 
	D_m.
\]
If $n\geq 2$ and $m\geq 0$, then the first step of $\mu$ is an up step and the entire path can be decomposed into some number of shorter primary paths.
Suppose that the shape of  $\mu$ is $(m,d,k)$ and
\begin{equation*}
	\mu = U_k^{m,d} \, \mu^{(1)} \mu^{(2)} \cdots \mu^{(d)}.
\end{equation*}
The weight of the first step is $v_1$ which is an integer from $\{ 1,2,\ldots,$ $ |\mathcal H_{\mathcal V}(m,d,k)| \}$, by \eqref{eq:weight}. Suppose that $(h,\lambda)$ is the $v_1$th pair from $\mathcal H_{\mathcal V}(m,d,k)$.
Let $\lambda = (\lambda_0,\lambda_1,\ldots,\lambda_d)$. We set
\begin{equation}
	f_{m,n}((\mu,v)) 
	\overset{\mathrm{def}}{=\joinrel=} 
	U_h \,  V^{\lambda_0}\, f(\mu^{(1)}) V^{\lambda_1} f(\mu^{(2)}) V^{\lambda_2}
	\cdots f(\mu^{(d)}) V^{\lambda_d},
\end{equation}
where $f(\mu^{(i)}) \equiv f_{m_i,n_i}((\mu^{(i)},v^{(i)}))$ for some $m_i,n_i$ depending on $\mu^{(i)}$, and $v^{(i)}$ is the proper part of the sequence $v$, for $i=1,2,\ldots,d$.

\begin{lemma}
\label{lem:tov}
If $(\mu,v)\in\mathcal W_{\mathcal L}(m,n)$ then $f_{m,n}((\mu,v))\in\mathcal P_{\mathcal V}(m,n)$, for any $m\geq 0$ and $n\geq 1$.
\end{lemma}

\begin{proof}
Suppose that $n\geq 2$. Observe that $f$ changes step $U_k$ to $U_h\in\mathcal V$, where $h\geq k$, and adds $\lambda_0+\cdots+\lambda_d$ vertical steps $V$ to the original path. The pair $(h,\lambda)$ belongs to $\mathcal H_{\mathcal V}(m,d,k)$ which implies that $h-k = \lambda_0+\cdots+\lambda_d$. Using the induction on $n$ we show that $f_{m,n}((\mu,v))$ runs from $(0,0)$ to $(n,-m)$. The condition that $\lambda_d = 0$ for $m\geq 2$ ensures that the last step of the resulting path is not the vertical one.
\end{proof}

\begin{figure}[ht]
\begin{center}
	\includegraphics[width=80mm]{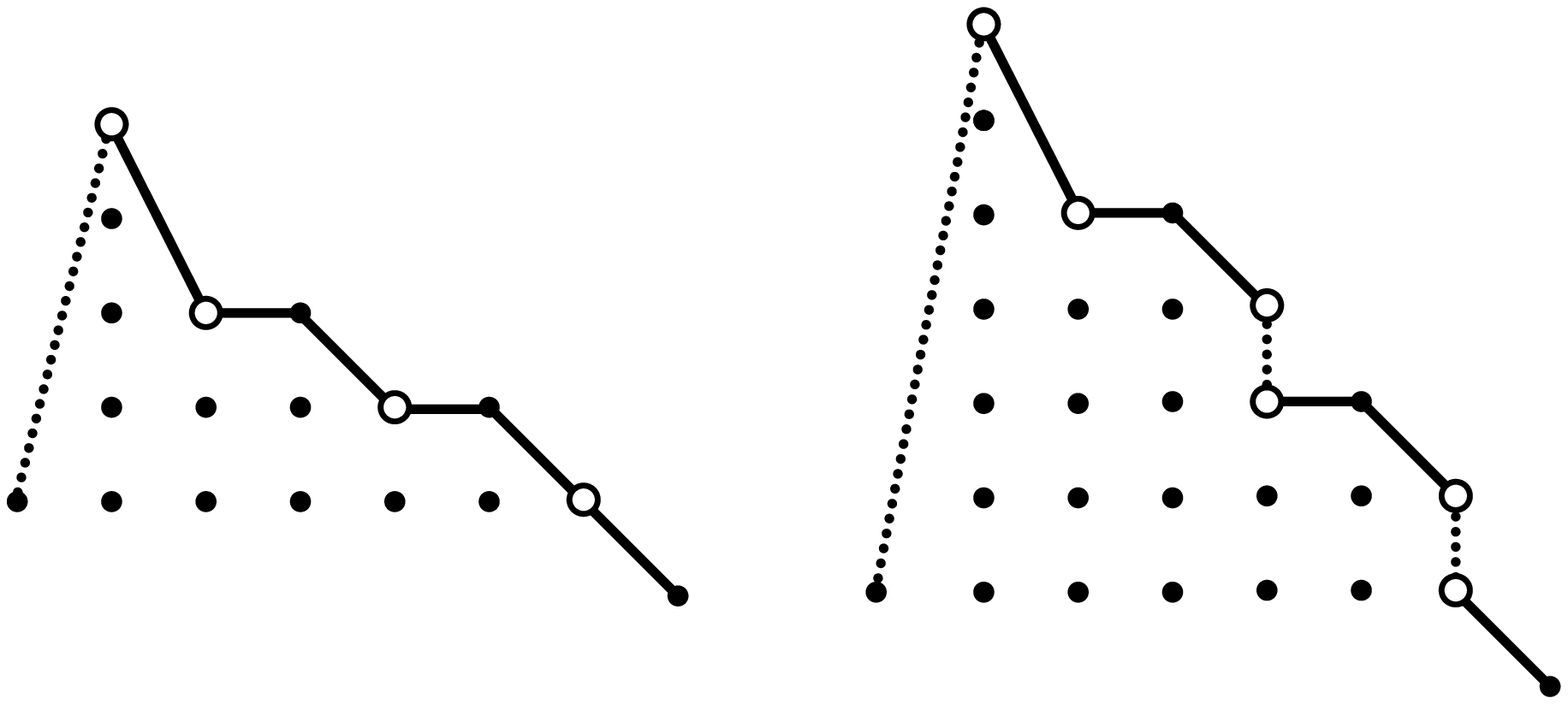}
	\caption{The map $f$ changes $U_4$ to $U_6$ and adds two vertical steps between certain subpaths. The map $g$ changes $U_6$ to $U_4$ and removes corresponding vertical steps.}
	\label{fig:mapf}
\end{center}
\end{figure}

\begin{example}
\label{exam:f}
Let us consider lattice paths which consist of steps from the set $\mathcal V=\{V, U_6, U_5,$ $ \ldots, U_0, D_1, D_2 \}$ and let $\mathcal L = \mathcal V\setminus \{V\}$.
Let $(\mu,v) \in \mathcal W_{\mathcal L}(1,7)$, where $\mu = U_4 D_2 U_0 D_1 U_0 D_1 D_1$, see Fig.~\ref{fig:3} and Fig.~\ref{fig:mapf}. Suppose that $v=(v_1,\ldots,v_7)$.
The path is decomposable into $U_4$ and $\mu^{(1)}, \ldots, \mu^{(4)}$, where
\[
	\mu^{(1)} = D_2,
	\quad
	\mu^{(2)} = U_0 D_1,
	\quad
	\mu^{(3)} = U_0 D_1,
	\quad
	\mu^{(4)} = D_1.
\]
By the definition of the weight function $w$, $v_1\in\{1,2,\ldots,|\mathcal H_{\mathcal V}(1,4,4)|\}$ and $|\mathcal H_{\mathcal V}(1,4,4)| = 21$, by Proposition~\ref{prop:h}.
Suppose that the $v_1$th pair from the set $\mathcal H_{\mathcal V}(1,4,4)$ is $(6,\lambda)$, where $\lambda=(0,0,1,1,0)$ is a composition of $2$ into $5$ parts. We have
\[
	f_{1,7}((\mu,v)) = U_6 V^{0} f(\mu^{(1)}) V^0 f(\mu^{(2)}) V^1 f(\mu^{(3)}) V^1 f(\mu^{(4)}) V^0.
\]
The path $U_6 V^{0} \mu^{(1)} V^0 \mu^{(2)} V^1 \mu^{(3)} V^1 \mu^{(4)} V^0$ is given in Fig.~\ref{fig:mapf}. The final path $f((\mu,v))$ for certain weight vector $v$ is given in Fig.~\ref{fig:1}.
\end{example}

\textbf{The map} $g_{m,n}:\mathcal P_{\mathcal V}(m,n)\to \mathcal W_{\mathcal L}(m,n)$, for $m\geq 0$, $n\geq 1$.

\vspace{0.2cm}
Let $\pi$ be a path from $\mathcal P_{\mathcal V}(m,n)$.
If $n=1$ and $m=0$, then $\pi = U_h V^h$ for certain $U_h\in\mathcal V$.
Suppose that the pair $(h,\lambda)$, where $\lambda = (h)$, is the $v_1$th pair in $\mathcal H_{\mathcal V}(0,0,0)$ is $(h,\lambda)$. Then we set
\[
	g_{0,1}(U_h V^h) 
	\overset{\mathrm{def}}{=\joinrel=}
	(U_0, (v_1)).
\]
If $n=1$ and $m=1$, then $\pi = S_h V^{h+1}$ for certain $S_h\in\mathcal V_{\geq -1}$. Suppose that $S_h$ is the $v_1$th step in $\mathcal V_{\geq -1}$. We set
\[
	g_{1,1}(U_h V^{h+1}) 
	\overset{\mathrm{def}}{=\joinrel=}
	(D_1, (v_1)).
\]
If $n=1$ and $m\geq 2$, then $\pi = D_m$. We set
\[
	g_{m,1}(D_m) 
	\overset{\mathrm{def}}{=\joinrel=}
	(D_m, (1)).
\]
If $n\geq 2$ then the first step of $\pi$ is an up step and the entire path $\pi$ can be decomposed into some number of shorter primary paths.
 Suppose that the shape of $\pi$ is $(m,d,k)$ and it can be decomposed as
\[
	\pi = U_h V^{\lambda_0} \pi^{(1)} V^{\lambda_1} \pi^{(2)} V^{\lambda_2} \cdots
	\pi^{(d)} V^{\lambda_d},
\]
Let $\lambda = (\lambda_0,\ldots,\lambda_d)$. Suppose that $(h,\lambda)$ is the $v_1$th pair in $\mathcal H_{\mathcal V}(m,d,k)$. We set
\begin{equation}
	g_{m,n}(\pi) \overset{\mathrm{def}}{=\joinrel=} 
	\left( U_k \, g\big(\pi^{(1)}\big)  g\big(\pi^{(2)}\big) \cdots g\big(\pi^{(d)}\big) , v \right),
\end{equation}
where $v=(v_1,v_2,\ldots,v_r)$ and $v_2,\ldots,v_r$ depend on $g(\pi^{(1)}),\ldots,g(\pi^{(d)})$.

\begin{lemma}
If $\pi\in\mathcal P_{\mathcal V}(m,n)$ then $g_{m,n}(\pi)\in\mathcal W_{\mathcal L}(m,n)$, for any $m\geq 0$ and $n\geq 1$.
\end{lemma}

\begin{proof}
Suppose that $n\geq 2$. As in the proof of Lemma~\ref{lem:tov}, the function $g$ changes step $U_h$ to $U_k\in\mathcal L$ and removes exactly $h-k$ vertical steps from the original path. Thus the resulting path ends at the same lattice point as the original one does. Using the induction we also show that the resulting path does not contain vertical steps.
\end{proof}

\begin{example}
\label{exam:g}
As in the previous example, let us consider lattice paths which consist of steps from the set $\mathcal V=\{V, S_6, S_5, \ldots S_{-2} \}$ and let $\mathcal L = \mathcal V\setminus \{V\}$.
Let $\pi \in \mathcal P_{\mathcal V}(1,7)$ be the path given in Fig.~\ref{fig:1}.
The decomposition of $\pi$ is given in Example~\ref{exam:decompos}. The shape of $\pi$ is $(1,4,4)$ and $\lambda=(0,0,1,1,0)$. In Example~\ref{exam:f}, we assume that the pair $(6,\lambda)$ is the $v_1$th element from $\mathcal H_{\mathcal V}(1,4,4)$. Thus $g_{1,7}(\pi) = (\mu,v)$, where $v=(v_1,\ldots,v_7)$ and
$\mu = U_4 \, g(\pi^{(1)}) \, g(\pi^{(2)}) \cdots g(\pi^{(4)})$.
The final path $g(\pi)$ for certain weight vector $v$ is given in Fig.~\ref{fig:3}.
\end{example}

\begin{lemma}
\label{lem:biject}
We have $f_{m,n}^{-1} = g_{m,n}$ for any $m\geq 0$ and $n\geq 1$.
\end{lemma}

\begin{proof}
We need to show that $g(f((\mu,v))) = (\mu,v)$ for any weighted path $(\mu,v)\in\mathcal W_{\mathcal L}(m,n)$ and $f(g(\pi)) = \pi$ for any $\pi\in\mathcal P_{\mathcal V}(m,n)$.
Simple verification shows that the claim is true for $n=1$. Let us prove the first statement for $n\geq 2$. Let $\mu = U_{k}^{m,d} \mu^{(1)} \cdots \mu^{(d)}$ and $v=(v_1,\ldots,v_n)$. Assume that $(h,\lambda)$ is the $v_1$th pair from $\mathcal H_{\mathcal V}(m,d,k)$ and $\lambda=(\lambda_0,\ldots,\lambda_d)$. On the one hand, by the definition of $f$, the path $\pi = f_{m,n}((\mu,v))$ can be rewritten as
\[
	\pi = 
	U_h \,  V^{\lambda_0}\, f(\mu^{(1)}) V^{\lambda_1} f(\mu^{(2)}) V^{\lambda_2}
	\cdots f(\mu^{(d)}) V^{\lambda_d},
\]
On the other hand, $\pi$ can be decomposed as 
\[
	\pi = U_h V^{\rho_0} \pi^{(1)} V^{\rho_1} \pi^{(2)} V^{\rho_2} \cdots
	\pi^{(t)} V^{\rho_t},
\]
for some primary $\mathcal V$-paths $\pi^{(1)},\ldots, \pi^{(t)}$. 
First, we need to show two statements: (i) $d=t$, $\pi^{(i)} = f(\mu^{(i)})$ for $i=1,2,\ldots,d$, and (ii) $(\lambda_0,\ldots,\lambda_d) = (\rho_0,\ldots,\rho_d)$.
The first one is due to the definition of the function $f$. The resulting path $\pi = f_{m,n}((\mu,v))$ is the concatenation of paths $f(\pi^{(1)}), \ldots, f(\pi^{(d)})$, which are primary $\mathcal V$-paths, and some number (possibly zero) of vertical steps $V$ between these shorter primary subpaths. 
The second condition follows from the observation that 
any primary $\mathcal V$-path does not begin with a vertical step, thus $\lambda_i = \rho_i$ for $i=0,1,\ldots,d$.

Next, under the assumption at the beginning of the proof, $(h,\lambda)$ is the $v_1$th pair from $\mathcal H_{\mathcal V}(m,d,k)$. Thus, by the definition of the function $g$, the resulting path $g_{m,n}(\pi)$ is
\[
	g_{m,n}(\pi) = \Big( 
	U_{k}^{m,d} g(f(\mu^{(1)})) g(f(\mu^{(2)})) \cdots g(f(\mu^{(d)})),
	v
	\Big),
\]
where $v=(v_1,\ldots,v_n)$.
Using the induction on $n$ we show that $g(f(\mu^{(i)})) = \mu^{(i)}$ which ends the proof of the first statement.

The proof of $f(g(\pi)) = \pi$ goes in much the same way.
\end{proof}

\section{Relations between primary and free paths}
\label{sec:properties}

In this section we establish a relation between primary and free paths that contain vertical steps. Throughout the section, we fix $\mathcal S$ to be a subset of $\Omega_N\cup\{V\}$ such that $U_N \in \mathcal S$ and $N\geq 0$. Also, we set $\mathcal V = \mathcal S \cup \{V\}$ and $\mathcal N = \mathcal S \setminus \{V\}$.

Let $a$ be a sequence of $n$ integers $a_1,\ldots,a_n$. A partial sum of $a$ is the sum $a_1 + \cdots + a_k$, for $1\leq k\leq n$.
Raney \cite{Raney} shows that there is only one cyclic-shift $a' = (a_{k}, a_{k+1}, \ldots, a_n, a_1, \ldots, a_{k-1})$ of $a$ such that any partial sum of $a'$ is positive (see also \cite[p. 360]{Concrete}). This lemma appears in the literature also as the cycle lemma \cite{Dershowitz3}.
For our purposes, we reformulate this lemma.

\begin{lemma}[Raney lamma \cite{Raney}]
\label{lem:raney}
Let $b=(b_1,\ldots,b_n)$ be a sequence of integers whose sum is $-1$.
There is only one cyclic-shift $b'$ of $b$ such that every partial sum of $b'$ except the total sum is nonnegative.
\end{lemma}
\begin{proof}
Observe that if we rearrange terms of $b$ in reverse order and we negate them, then we obtain the sequence $(-b_n,-b_{n-1},\ldots,-b_1)$ whose sum is $+1$, and from the Raney lemma there is only one cyclic shift of such modified sequence which has the property that any its partial sum except the total sum is nonnegative.
\end{proof}

Therefore, the Raney lamma implies that
\begin{equation}
	\label{eq:formpn}
	|\mathcal P_{\mathcal N}(1,n)| = \frac{1}{n} |\mathcal F_{\mathcal N}(1,n)|,
	\qquad (n\geq 1).
\end{equation}
We extend this relation between $1$-primary and free $\mathcal N$-paths to the corresponding families of $\mathcal V$-paths with vertical steps.

\begin{theorem}
For any $n\geq 1$ we have
\begin{equation}
	\label{eq:Pv1n}
	|\mathcal P_{\mathcal V}(1,n)| = \frac{1}{n} 
	\big( 
	|\mathcal F_{\mathcal V}(1,n)| - |\mathcal F_{\mathcal V}(0,n)| 
	\big).
\end{equation}
\end{theorem}
\begin{proof}
Any path from $\mathcal P_{\mathcal V}(1,n)$ is represented as
$S_{a_1} V^{b_1} S_{a_2} V^{b_2} \cdots S_{a_n} V^{b_n}$ for some $a_1,\ldots,a_n$ depending on $\mathcal V$ and $b_1,\ldots,b_n\geq 0$. Let $\alpha=(a_1-b_1,a_2-b_2,\ldots,a_n-b_n)$. The total sum of members of $\alpha$ is $-1$ and any partial sum (except the total sum) is nonnegative. By the Raney lemma, the number of such sequences is equal to $1/n$ times the number of sequences $\beta = (c_1 - d_1,\ldots,c_n-d_n)$, where $d_1,\ldots,d_n\geq 0$, $c_1,\ldots,c_n$ depend on $\mathcal V$, and the total sum of elements of $\beta$ is $-1$. Observe that $\beta$ designates uniquely a free $\mathcal V$-path running from $(0,0)$ to $(n,-1)$ whose the first step is non-vertical. The number of sequences $\beta$ is $|\mathcal F_{\mathcal V}(1,n)| - |\mathcal F_{\mathcal V}(0,n)|$, which finishes the proof.
\end{proof}

\begin{theorem}
Let $n\geq 1$ and $m\in\mathbb Z$, then
\begin{align}
	\label{eq:FvN}
	|\mathcal F_{\mathcal V}(m,n) |
	&= \sum_{j=0}^{Nn+m} \binom{n+j}{j} |\mathcal F_{\mathcal N}(m-j,n)|,
	\\
	\label{eq:PvN}
	|\mathcal P_{\mathcal V}(1,n)| 
	&= \frac{1}{n} \sum_{j=0}^{Nn+1} \binom{n+j-1}{j}
	|\mathcal F_{\mathcal N}(1-j,n)|.
\end{align} 
\end{theorem}
\begin{proof}
First, we show \eqref{eq:FvN}. The number of vertical steps in any path from $\mathcal F_{\mathcal V}(m,n)$ is an integer from $\{0,1,\ldots,Nn+m\}$. Therefore, we partition the family $\mathcal F_{\mathcal V}(m,n)$ into pairwise disjoint subfamilies $A_0,A_1,\ldots,$ $A_{Nn+m}$, such that $A_j$ contains these paths whose number of vertical steps is $j$. 
To calculate the size of $A_j$, observe that adding $j$ vertical steps to any free $\mathcal N$-path (without vertical steps) running from $(0,0)$ to $(n,j-m)$ we obtain a free path from $\mathcal F_{\mathcal V}(m,n)$.
Any such path has $n$ non-vertical steps $S_k$ and those $j$ vertical steps may be added between them on $s$ ways, where $s$ is the number of solutions of $a_0 + a_1 + \cdots + a_n = j$, where $a_0,\ldots,a_n\geq 0$.
Therefore, the size of $A_j$ is $\binom{n+j}{j}$ times the size of $\mathcal F_{\mathcal N}(m-j,n)$.

The second equality \eqref{eq:PvN} follows directly from \eqref{eq:Pv1n} together with \eqref{eq:FvN}. That is,
\begin{align*}
	|\mathcal P_{\mathcal V}(1,n)| 
	&= \frac{1}{n} \left( 
		\sum_{j=0}^{Nn+1} \binom{n+j}{j} |\mathcal F_{\mathcal N}(1-j,n)|
		-
		\sum_{j=0}^{Nn} \binom{n+j}{j} |\mathcal F_{\mathcal N}(0-j,n)|
	\right).
\end{align*}
Changing the range summation of the second sum and using the recurrence relation for binomial coefficients we obtain the required formula.
\end{proof}

Suppose that $S$ is a step from $\mathcal S$. Let $\#Steps(S \in \mathcal P_{\mathcal S}(1,n))$ denote the total number of occurrences of steps $S$ in the set of all paths from $\mathcal P_{\mathcal S}(1,n)$, and let $\#Steps(\mathcal P_{\mathcal S}(1,n))$ denote the total number of all steps in $\mathcal P_{\mathcal S}(1,n)$.

\begin{theorem}
\label{th:prop1}
Let $n\geq 1$, then
\begin{subequations}
\begin{align}
	\#Steps(V \in \mathcal P_{\mathcal S}(1,n)) &= |\mathcal F_{\mathcal S}(0, n)|,
	\\
	\#Steps(S_k \in \mathcal P_{\mathcal S}(1,n)) &= |\mathcal F_{\mathcal S}(1+k, n-1)|.
\end{align}
\end{subequations}
\end{theorem}

\begin{proof}
Let $S$ be a fixed step from $\mathcal S$ and let us introduce the temporary notation $\mathcal F$ for $\mathcal F_{\mathcal S}(0,n)$ if $S=V$ or $\mathcal F_{\mathcal S}(1+k,n-1)$ if $S=S_k$ for certain $k\in\mathbb Z$. Further, by a  \emph{level} we mean a line $y=l$ for any $l\in\mathbb Z$.
Take $\pi\in \mathcal P_{\mathcal V}(1,n)$ and suppose that $\pi$ has exactly $d$ steps $S$ and $d\geq 1$. Let $1\leq p\leq d$, then
\begin{equation}
	\label{eq:freedecompos}
	\pi = \underbrace{\pi^{(1)} S\, \pi^{(2)} S \cdots S\, \pi^{(p-1)} S\, \pi^{(p)}}_{\alpha}
	 S 
	 \underbrace{\pi^{(p+1)} S \cdots  \pi^{(d)} S\, \pi^{(d+1)}}_{\beta},
\end{equation}
for certain possibly empty subpaths $\pi^{(1)},\pi^{(2)},\ldots,\pi^{(d+1)}$. We define a function $\phi$ from the set of all occurrences of steps $S$ in paths from $\mathcal P_{\mathcal S}(1,n)$ to the set of paths from $\mathcal F$ as follows
\begin{equation}
	\label{eq:freedecompos1}
	\phi(\pi,p) = \beta\, \alpha,
\end{equation}
where $\alpha$ and $\beta$ are subpaths of $\pi$ defined as in \eqref{eq:freedecompos}.
To show that $\phi$ is a bijection we need to show that $\phi(\pi,p)$ is a path from $\mathcal F$ and any path from $\mathcal F$ is decomposable as \eqref{eq:freedecompos1} for certain uniquely determined $p\geq 1$. Then the procedure $\phi$ is reversible and $\phi$ is a bijection.

First, observe that $\phi(\pi,p)$ removes only one step $S$ from $\pi$ which implies that the result is a free path from $\mathcal F$.
Next, suppose that $(x,y)$ is the leftmost point of $\phi(\pi,p)$ such that $y$ is the minimal level that the path reaches. We prove that the path $\phi(\pi,p)$ reaches $(x,y)$ exactly after the last step of $\beta$ in \eqref{eq:freedecompos1}.
Recall that $\pi$ is a primary $\mathcal S$-path running from $(0,0)$ to $(n,-1)$ for which only its ending point lies below the $x$-axis. Thus $\pi$ reaches the lowest level exactly after part $\pi^{(d+1)}$. It follows that $\alpha$ is a path that does not go below the $x$-axis. On the other hand, only the ending point of $\beta$ reaches the lowest level. It follows that $p-1$ is the number of steps $S$ of $\phi(\pi,p)$ that lie to the right from $(x,y)$.

Let $\gamma$ be a free $\mathcal S$-path $\gamma$ from $\mathcal F$ and $\gamma = \beta \alpha$ such that the last point of the subpath $\beta$ lies at the left-most minimal level that $\gamma$ reaches. Then we set $\phi^{-1}(\gamma)$ to be the pair $(\alpha S \beta, p)$, where $p$ is the number of steps $S$ in $\alpha$ plus one.

\end{proof}

\begin{example}
\label{examp:prop1}
Let $\pi$ be a path from Fig.~\ref{fig:1} and $S=V$. The path $\phi(\pi,2)$ is given in Fig.~\ref{fig:prop1}.
\end{example}

\begin{figure}[ht]
\begin{center}
	\includegraphics[width=35mm]{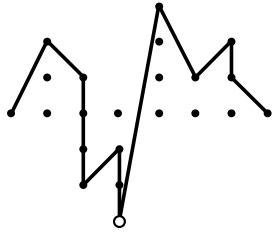}
	\caption{A free $\mathcal V$-path running from $(0,0)$ to $(7,0)$ which has one vertical step that lie to the right from the lowest point drawn by the open circle.}
	\label{fig:prop1}
\end{center}
\end{figure}

\begin{theorem}
\label{th:prop2}
For $n\geq 1$, we have $\#Steps(\mathcal P_{\mathcal S}(1,n)) = |\mathcal F_{\mathcal S}(1,n)|$.
\end{theorem}

\begin{proof}
We show a bijection $\psi$ between the set of all steps in paths from $\mathcal P_{\mathcal S}(1,n)$ and the set of paths from $\mathcal F_{\mathcal S}(1,n)$. 
Take a path $\mu$ from $\mathcal P_{\mathcal S}(1,n)$ and suppose that $\mu = \mu_1 \cdots \mu_r$. Let $k \in \{1,2,\ldots,r\}$, then we set
\[
	\psi(\mu,k) = \mu_k \, \mu_{k+1} \cdots \mu_r \, \mu_1 \, \mu_2 \cdots \mu_{k-1}.
\]
It is clear that $\psi(\mu,k) \in \mathcal F_{\mathcal S}(1,n)$.
Next, we show a map $\zeta$ from $\mathcal F_{\mathcal S}(1,n)$ to the set of all steps in paths from $\mathcal P_{\mathcal S}(1,n)$. Let $\pi$ be a path from $\mathcal F_{\mathcal S}(1,n)$ and $\pi = \pi_1 \cdots \pi_r$. 
Let us represent $\pi$ as the sequence $\hat s=(\hat s_1,\hat s_2,\ldots,\hat s_r)$ of integers according to the rule
\[
	\hat s_i = \left\{
	\begin{array}{rl}
	k & \mathrm{if} \,\, \pi_i = S_k, \\
	-1 &  \mathrm{if} \,\, \pi_i= V. \\
	\end{array}
	\right.
\]
The sum of terms of the sequence $\hat s$ is $-1$. Therefore, the modified Raney lemma (Lemma~\ref{lem:raney}) implies that there is only one cyclic-shift $s=(s_1,\ldots,s_r)$ of $\hat s$ that every its partial sum except the total sum is nonnegative. Moreover, this cyclic-shift $s$ determines uniquely an index $k$ such that the cyclic-shift 
$(s_k, s_{k+1},\ldots, s_r, s_1, \ldots,s_{k-1})$
of $s$ is the original sequence $\hat s$. Now, if we change back terms of the sequence $s$ into steps according to the above rule, we obtain a primary $\mathcal S$-path $\mu$.
This implies that with any free path $\pi$ from $\mathcal F_{\mathcal S}(1,n)$ we have associated uniquely a primary path $\mu$ from $\mathcal P_{\mathcal S}(1,n)$ and an index $k$ such that $\phi(\mu,k) = \pi$.
\end{proof}

\section{Enumerative results}
\label{sec:enumer}

As in previous sections, we fix $\mathcal S$ to be a subset of $\Omega_N\cup\{V\}$ such that $U_N \in \mathcal S$. Also, we set $\mathcal V = \mathcal S \cup \{V\}$ and $\mathcal N = \mathcal S \setminus \{V\}$.
Recall that we denote by $\mathcal P_{\mathcal S}(m,n)$ the set of all $m$-primary $\mathcal S$-paths running from $(0,0)$ to $(n,-m)$, and by $\mathcal F_{\mathcal S}(m,n)$ the set of all free $\mathcal S$-paths running from $(0,0)$ to $(n,-m)$ without further restriction. In this section we derive formulas for the sizes of these families.

\subsection{General case}

First, we consider the case where the set of steps $\mathcal V$ may contain infinitely many down steps.
The number of free $\mathcal V$-paths running from $(0,0)$ to $(n,-m)$ satisfies the following recurrence relation
\begin{equation}
	\label{eq:recurfree}
	|\mathcal F_{\mathcal V}(m,n)| 
	= |\mathcal F_{\mathcal V}(m-1,n)|
	+ \sum_{S_k\in\mathcal V} |\mathcal F_{\mathcal V}(m+k,n-1)|
\end{equation}
with initial conditions $|\mathcal F_{\mathcal V}(-Nn,n)| = |\mathcal F_{\mathcal V}(r,0)| = 1$, for $m \geq -Nn$ and $r,n\geq 0$. For $n<0$ or $m < -Nn$ the number of such paths is zero.
Thus even $\mathcal V$ has infinitely many down steps, the sum on the right-hand side of \eqref{eq:recurfree} is finite.

Let us define a bivariate generating function in the sense that
\[
	F_{\mathcal V}(x,y) = \sum_{m\geq 0} \sum_{n\geq 0} |\mathcal F_{\mathcal V}(m-Nn,n)|x^n y^m.
\]
\begin{proposition} We have
\begin{align}
	\label{eq:generfree}
	F_{\mathcal V}(x,y) &= \Big( 1 - y - x \sum_{S_k\in\mathcal V}y^{N-k} \Big)^{-1}.
\end{align}
\end{proposition}
\begin{proof}
Applying standard methods of generatingfunctionology \cite{wilfgener} to the recurrence relation \eqref{eq:recurfree} one can show that
\[
	F_{\mathcal V}(x,y) = 1 + y F_{\mathcal V}(x,y) + x \sum_{S_k\in\mathcal S} y^{N-k} F_{\mathcal V}(x,y)
\]
which implies \eqref{eq:generfree} immediately.
\end{proof}

\begin{proposition}
\label{prop:numfree}
Let $n\geq 0$ and $m\in\mathbb Z$, then
\begin{subequations}
\begin{align}
	\label{eq:fznm1}
	|\mathcal F_{\mathcal N}(m,n) | 
	&= [y^{N n + m}] \Big( \sum_{S_k\in\mathcal N} y^{N-k} \Big)^n,
	\\
	\label{eq:fznm2}
	|\mathcal F_{\mathcal V}(m,n) | 
	&= [y^{N n + m}] \frac{1}{(1-y)^{n+1}}\Big( \sum_{S_k\in\mathcal N} y^{N-k} \Big)^n.
\end{align} 
\end{subequations}
\end{proposition}
\begin{proof}
Observe that any free $\mathcal N$-path running from $(0,0)$ to $(n,-m)$ can be represented as
$S_{N-a_1} S_{N-a_2}$ $ \cdots S_{N-a_n}$, where $a_1,a_2,\ldots,a_n$ are nonnegative integers whose sum is $N n + m$. Thus the number of paths from $\mathcal F_{\mathcal N}(m,n)$ is the coefficient of $y^{Nn+m}$ in the power series expansion of $(\sum_{S_k\in\mathcal N} y^{N-k})^n$, as claimed.
On the other hand, from \eqref{eq:generfree} we have
\begin{equation}
	\label{eq:genery}
	\sum_{m\geq 0} |\mathcal F_{\mathcal V}(m-Nn,n)| y^m
	= \frac{1}{(1-y)^{n+1}} \Big( \sum_{S_k\in\mathcal V} y^{N-k} \Big)^n,
\end{equation}
and the formula \eqref{eq:fznm2} follows. 
\end{proof}

\begin{proposition}
\label{prop:numprim}
Let $n\geq 1$, then
\begin{align}
	\label{eq:formprim}
	|\mathcal P_{\mathcal V}(1,n)| 
	&= \frac{1}{n} [y^{Nn+1}] \frac{1}{(1-y)^n}
	\Big( \sum_{S_k\in\mathcal N} y^{N-k} \Big)^n.
\end{align}
\end{proposition}
\begin{proof}
It follows from \eqref{eq:Pv1n} together with \eqref{eq:fznm1}. Namely,
\begin{align*}
	|\mathcal P_{\mathcal V}(1,n)| 
	&= \frac{1}{n} \left( 
	[y^{Nn+1}] A_n(y) - [y^{Nn}] A_n(y)
	\right)
	\\
	&= \frac{1}{n} \left( 
	[y^{Nn+1}] A_n(y)(1-y)
	\right),
\end{align*}
where $A_n(y) = (\sum_{S_k\in\mathcal N}y^{N-k} )^n (1-y)^{-n-1}$.
\end{proof}

\begin{corollary}
The expected number of vertical steps in a path from $\mathcal P_{\mathcal V}(1,n)$ is equal to
\begin{equation}
	n\cdot \frac{|\mathcal F_{\mathcal V}(0,n)|}{|\mathcal F_{\mathcal V}(1,n)| - |\mathcal F_{\mathcal V}(0,n)| }.
\end{equation}
\end{corollary}
\begin{proof}
The required number is the number of all vertical steps in the set of paths from $\mathcal P_{\mathcal V}(1,n)$ divided by the number of such paths. By Theorem~\ref{th:prop1}, this number is $|\mathcal F_{\mathcal V}(0,n)|/|\mathcal P_{\mathcal V}(1,n)|$. Applying \eqref{eq:Pv1n} we obtain the formula.
\end{proof}

\begin{corollary}
The expected number of steps in a path from $\mathcal P_{\mathcal V}(1,n)$ is equal to
\begin{equation}
	\label{eq:expnumsteps}
	n \left( 1 + \frac{ |\mathcal F_{\mathcal V}(0,n)|}{|\mathcal F_{\mathcal V}(1,n)| - |\mathcal F_{\mathcal V}(0,n)| } \right).
\end{equation}
\end{corollary}
\begin{proof}
The required number is the total number of steps in the set of paths from $\mathcal P_{\mathcal V}(1,n)$ divided by the number of such paths. By Theorem~\ref{th:prop2}, this number is $|\mathcal F_{\mathcal V}(1,n)|/|\mathcal P_{\mathcal V}(1,n)|$. Applying \eqref{eq:Pv1n} we obtain the formula.
\end{proof}

\subsection{Finite set of steps}
Throughout the section we assume that the sets of possible steps $\mathcal S, \mathcal V, \mathcal N$, and $\mathcal L$ are finite and $K$ is the maximal integer such that $D_K$ belongs to those sets. Recall that $\mathcal L$ is defined in \eqref{eq:defN} with respect to $\mathcal V$. 
It is worth pointing out that a unified enumerative and asymptotic theory of lattice paths consisting of steps from $\mathcal N$ is developed by Banderier and Flajolet \cite{Banderier2002} and is associated with the so-called kernel method.

\begin{proposition}
\label{prop:P0}
If $K=1$, then for $n\geq 0$ we have
\begin{align}
	|\mathcal P_{\mathcal V}(0,n)| = \sum_{j=0}^n (-1)^{n-j} |\mathcal P_{\mathcal V}(1,j)|.
\end{align}
\end{proposition}
\begin{proof}
If $K=1$ then there is no step $D_p$ in $\mathcal V$ such that $p>1$. Thus the last step of any path from $\mathcal P_{\mathcal V}(1,n)$ is either $V$ or $D_1$. It follows that
$|\mathcal P_{\mathcal V}(1,n)| = |\mathcal P_{\mathcal V}(1,n-1)| + |\mathcal P_{\mathcal V}(0,n)|$. Moving $|\mathcal P_{\mathcal V}(1,n-1)|$ to the left-hand side we obtain a recurrence relation for $|\mathcal P_{\mathcal V}(0,n)|$.
Iterating the above gives the required sum.
\end{proof}

\begin{proposition}
If $K=m$ and $K\geq 2$, then for $n\geq 1$ we have
\[
	|\mathcal P_{\mathcal V}(K,n)| = |\mathcal P_{\mathcal V}(0,n-1)|.
\]
\end{proposition}
\begin{proof}
This follows from the observation that the last step of any $m$-primary $\mathcal V$-path running from $(0,0)$ to $(n,-m)$, where $m=K$, is $D_K$. Removing this step we obtain $0$-primary $\mathcal V$-path running from $(0,0)$ to $(n-1,0)$.
\end{proof}

Let us define two ordinary generating functions
\begin{equation}
	P_{\mathcal S,m}(x) = \sum_{n\geq 0} |\mathcal P_{\mathcal S}(m,n)| x^n,
	\quad
	W_{\mathcal S,m}(x) = \sum_{n\geq 0} |\mathcal W_{\mathcal S}(m,n)| x^n.
\end{equation}
For simplicity of notation, we write $F_m(x)$ instead of $F_{\mathcal S,m}(x)$ for fixed $\mathcal S$.

\begin{proposition}
\label{prop:gener}
Let $1\leq m\leq K$, then
\begin{equation}
\begin{split}
	\label{eq:gener}
	P_{\mathcal N, 0}(x) 
	&= 1 + \delta_0 x P_{\mathcal N,0}(x)
	+ x P_{\mathcal N,0}(x) \sum_{S_k\in\mathcal N_1}
	\sum_{d=1}^{k} \sum_{M}
	\prod_{j=1}^d 
	(P_{\mathcal N, m_j}(x)-1),
	\\
	P_{\mathcal N, m}(x)
	&= 1 + \delta_m x + x 
	\sum_{U_k\in\mathcal N_0} 
	\sum_{d=1}^{k+1}
	\sum_{M}
	\prod_{j=1}^d 
	(P_{\mathcal N, m_j}(x)-1),
\end{split}
\end{equation}
where $\delta_m = 1$ if $S_{-m}\in\mathcal N$, and $\delta_m = 0$ if $S_{-m}\notin \mathcal N$, and the summation range $M$ is over all solutions of $m_1+\cdots+m_d=k+m$ such that $1 \leq m_1,\ldots,m_{d-1}\leq K$ and $\max(m,1)\leq m_d\leq K$.
\end{proposition}

\begin{proof}
It follows from the decomposition of an $m$-primary $\mathcal N$-path.
By convention, we have one path of length zero. If $S_{-m}\in\mathcal N$ then we have one path of length one. Let $n\geq 2$ and take any path from $\mathcal P_{\mathcal N}(m,n)$. The first step of this path is an up step, let say $U_k$.
If $m\geq 1$ then the entire path $\mu$ is decomposable into $U_k$ and some number, let say $d$, of shorter and nonempty primary $\mathcal N$-paths $\mu^{(1)}, \ldots, \mu^{(d)}$. Suppose that $\mu^{(i)}\in \mathcal P_{\mathcal N}(m_i,n_i)$, then the numbers $m_1,\ldots,m_d$ are positive integers no greater than $K$. Further, $m_d$ is no smaller than $m$. Finally, if $m=0$, then the path is decomposable as above with some number (possibly zero) of additional $0$-primary $\mathcal N$-paths.
\end{proof}

\begin{proposition}
Let $1\leq m\leq K$, then
\begin{equation}
\begin{split}
	W_{\mathcal L, 0}(x) 
	&= 1 + \delta_0 x W_{\mathcal L,0}(x)
	+ x W_{\mathcal L,0}(x) \sum_{k=1}^N
	\sum_{d=1}^{k} 
	|\mathcal H_{\mathcal V}(0,d,k)|	
	\sum_{M}
	\prod_{j=1}^d 
	(W_{\mathcal L, m_j}(x)-1),
	\\
	W_{\mathcal L,m}(x)
	&= 1 + \delta_m x + x \sum_{k=0}^N \sum_{d=1}^{k+1}  
	|\mathcal H_{\mathcal V}(m,d,k)|   
	\sum_{M}
	\prod_{j=1}^d (W_{\mathcal L, m_j}(x)-1),
\end{split}
\end{equation}
where $\delta_m = |\mathcal L_{-m}|$ if $m\in\{0,1\}$, $\delta_m = 1$ if $D_m\in\mathcal L$, and $\delta_m = 0$ if $D_m\notin\mathcal L$, for $m\geq 2$. Further, the summation range $M$ is over all solutions of $m_1+\cdots+m_d=k+m$ such that $1 \leq m_1,\ldots,m_{d-1}\leq K$ and $\max(m,1)\leq m_d\leq K$.
\end{proposition}

\begin{proof}
The set of steps $\mathcal L$ contains steps $U_N, U_{N-1},\ldots, U_0,D_1$, thus the range summation of the first sum is over $k$ from one (for $m=0$) or zero (for $m\neq 1$) up to $N$. The weight of any step $U_k^{m,d}$ is the size of $\mathcal H_{\mathcal V}(m,d,k)$. Thus substituting that numbers into the functional equation \eqref{eq:gener} we obtain the formula. 
\end{proof}

\begin{proposition}
\label{prop:generV}
Let $1\leq m\leq K$, then
\begin{equation}
\begin{split}
	P_{\mathcal V, 0}(x) 
	&= 1 + \delta_0 x P_{\mathcal V,0}(x)
	+ x P_{\mathcal V,0}(x) \sum_{k=1}^N
	\sum_{d=1}^{k} 
	|\mathcal H_{\mathcal V}(0,d,k)|	
	\sum_{M}
	\prod_{j=1}^d 
	(P_{\mathcal V, m_j}(x)-1),
	\\
	P_{\mathcal V,m}(x)
	&= 1 + \delta_m x + x \sum_{k=0}^N \sum_{d=1}^{k+1}  
	|\mathcal H_{\mathcal V}(m,d,k)|   
	\sum_{M}
	\prod_{j=1}^d (P_{\mathcal V, m_j}(x)-1),
\end{split}	
\end{equation}
where $\delta_m = |\mathcal V_{\geq -m}|$ if $m\in\{0,1\}$, $\delta_m = 1$ if $D_m\in\mathcal V$, and $\delta_m = 0$ if $D_m\notin\mathcal V$, for $m\geq 2$.
Further, the summation range $M$ is over all solutions of $m_1+\cdots+m_d=k+m$ such that $1 \leq m_1,\ldots,m_{d-1}\leq K$ and $\max(m,1)\leq m_d\leq K$.
\end{proposition}

\begin{proof}
By Theorem~\ref{th:biject}, the size of $\mathcal W_{\mathcal L}(m,n)$ is equal to the size of $\mathcal P_{\mathcal V}(m,n)$, thus $P_{\mathcal V,m}(x)$ and $W_{\mathcal L,m}(x)$ are the same generating functions.
\end{proof}

\subsection{Riordan arrays}
\label{sec:riordan}

The Riordan group \cite{Shapiro3, Sprugnoli1} is a set of infinite lower-triangular matrices defined as follows. A \emph{proper Riordan array} is a couple $(g(x), f(x))$, where $g(x) = \sum_{n\geq 0} g_n x^n$ with $g_0\neq 0$ and $f(x) = \sum_{n\geq 1} f_n x^n$ with $f_1\neq 0$. With the proper Riordan array we associate the matrix, denoted by $(g,f)$, whose $(i,j)$th element is given by $[x^i] g(x) f(x)^j$, for $i,j\geq 0$.

\begin{proposition}
\label{prop:riordan}
The array 
\begin{equation}
	\label{eq:riordan}
	D_{\mathcal V} = \left( \frac{1}{1-y}, \frac{y}{1-y} \sum_{S_k\in\mathcal V} y^{N-k} \right)
\end{equation}
is the proper Riordan array, whose $(i,j)$th element, denoted by $d_{i,j}$, is the number of free $\mathcal V$-paths running from $(0,0)$ to $(j,(N+1)j-i)$. That is, $d_{i,j} = |\mathcal F_{\mathcal V}(i-(N+1)j,j)|$.
\end{proposition}
\begin{proof}
By \eqref{eq:genery}, since $g(x)=1/(1-x)$ and $f(x) = x(\sum_{S_k\in\mathcal V}x^{N-k})/(1-x)$, we conclude that $(g,f)$ is the proper Riordan array, and 
\[
	d_{i,j} = [y^i] \frac{1}{1-y} \left( \frac{y}{1-y}\sum_{S_p\in\mathcal S}y^{N-p} \right)^j
	= |\mathcal F_{\mathcal V}(i-(N+1)j,j)|.
\]
\end{proof}

\begin{corollary}
Let $a(x)$ be the generating function of the sequence $(a_n)_{n\geq 0}$. Then
\[
	\sum_{k\geq 0} d_{n,k} a_k = [y^n] \frac{1}{1-y} a\left( \frac{y}{1-y} \sum_{S_k\in\mathcal V} y^{N-k} \right),
\]
where $d_{n,k} = |\mathcal F_{\mathcal V}(n-(N+1)k,k)|$. That is, 
$d_{n,k}$ is the number of free $\mathcal S$-paths running from $(0,0)$ to $(k, Nk+k-n)$.
\end{corollary}
\begin{proof}
It follows directly from the properties of the Riordan arrays, see e.g. Sprugnoli \cite[Th.1.1]{Sprugnoli1}.
\end{proof}

\begin{example}
If $\mathcal V = \{V, U_1, U_0 \}$, then
\[
	D_{\mathcal V} = \left( \frac{1}{1-y}, \frac{y + y^2}{1-y}  \right),
	\qquad
	d_{i,j} = [y^{i-j}] \frac{\left( 1+y \right)^j}{(1-y)^{j+1}} 
	= D(j,i-j),
\]
where $D(i,j)$ is the $(i,j)$th Delannoy number \cite{Delannoy, stanton}. Note that $D(i,j)$ is the number of paths running from $(0,0)$ to $(i,j)$ consisting of steps $(1,0)$, $(1,1)$, and $(0,1)$. Such paths are called Delannoy paths in the literature.
By Proposition~\ref{prop:riordan}, we have
\[
	d_{i,j} = |\mathcal F_{\mathcal V}(i-2j,j)| = D(j,i-j).
\]
Therefore, there is a bijection between free $\mathcal V$-paths running from $(0,0)$ to $(j,j-i)$ and Delannoy paths running from $(0,0)$ to $(j,i)$. See Section~\ref{sec:delan} for more details.
\end{example}

\section{Examples of lattice paths with vertical steps}
\label{sec:examples}

In this section we present five examples of lattice paths with vertical steps for which we apply results obtained in the previous sections.

\subsection{The first example}
\label{sec:firstexample}

Let $\mathcal A=\{V, U_N, U_{N-1}, \ldots, U_0, D_1, \ldots, D_K \}$ for fixed $N\geq 0$ and $K\geq 1$. Let us define the corresponding set of lattice steps without vertical step $V$. That is, $\mathcal L = \mathcal A \setminus \{V\}$.
By Theorem~\ref{th:biject}, for $m\geq 0$, and $n\geq 1$, we have
\[
	|\mathcal P_{\mathcal A}(m,n)| = \sum_{\mu\in\mathcal P_{\mathcal L}(m,n)} w(\mu),
\]
if the weight function $w$ over steps from $\mathcal L$-paths is defined as follows
\begin{align}
	w_{}(D_p) = \left\{ 
\begin{array}{cl}
	N+2 & \mathrm{if}\,\, p = 1, \\ 
	1 & \mathrm{if}\,\, p \geq 2,
\end{array}		
	\right.,
	\quad
	w_{}(U_k^{m,d}) = \binom{N-k+d+1-\epsilon_m}{N-k},
\end{align}
where $\epsilon_m = 0$ if $m\in\{0,1\}$ and $\epsilon_m = 1$ if $m\geq 2$.

\begin{corollary} 
For $m\geq 0$ and $n\geq 1$, we have
\begin{subequations}
\begin{align}
	|\mathcal F_{\mathcal A}(m,n) |
	&=
	\sum_{k=0}^{\lfloor \frac{N n + m}{N+K+1} \rfloor}
	(-1)^k \binom{n}{k} \binom{n(N+2)-k(N+K+1)+m}{2n},
	\\
	|\mathcal P_{\mathcal A}(1,n)| 
	&= \frac{1}{n}
	\sum_{k=0}^{\lfloor \frac{N n + 1}{N+K+1} \rfloor}
	 (-1)^k \binom{n}{k} \binom{n(N+2)-k(N+K+1)}{2n-1}.
\end{align}
\end{subequations}
\end{corollary}
\begin{proof}
By Proposition~\ref{prop:numfree} and Proposition~\ref{prop:numprim}, we have
\begin{align*}
	|\mathcal F_{\mathcal A}(m,n) | 
	= [z^{N n + m}] 
	\frac{(1-z^{N+K+1})^{n}}{(1-z)^{2n+1}},
	\quad
	|\mathcal P_{\mathcal A}(1,n)| 
	= \frac{1}{n} [z^{Nn+1}] \frac{(1 - z^{N+K+1})^n}{(1-z)^{2n}}.
\end{align*} 
From the binomial theorem we derive
\begin{equation}
	\label{eq:expans1}
	\frac{(1-z^{A})^B}{(1-z)^{C}} = \sum_{n\geq 0} \sum_{k=0}^{\lfloor \frac{n}{A} \rfloor}  (-1)^k \binom{B}{k} \binom{C+n-kA-1}{C-1} z^n.
\end{equation}
for any integers $A,B,C\geq 0$. Substituting required parameters we obtain the formulas.
\end{proof}

\begin{corollary}
\label{cor:gener1}
Let $P_{m}(x) = P_{\mathcal A,m}(x)$ and $1\leq m\leq K$, then
\begin{equation}
\begin{split}
	P_{0}(x) 
	&= 1 + \delta_0 x P_{0}(x)
	+ x P_{0}(x) \sum_{k=1}^N
	\sum_{d=1}^{k} 
	\binom{N-k+d+1}{N-k}  
	\sum_{M}
	\prod_{j=1}^d 
	(P_{m_j}(x)-1),
	\\
	P_{m}(x) 
	&= 1 + \delta_m x + x \sum_{k=0}^N \sum_{d=1}^{k+1}  
	\binom{N-k+d+1-\epsilon_m}{N-k}  
	\sum_{M}
	\prod_{j=1}^d (P_{m_j}(x)-1),
\end{split}
\end{equation}
where $\delta_m = (N+1+m)$ if $m\in\{0,1\}$, $\delta_m = 1$ for $2\leq m\leq K$. Further $\epsilon_m = 0$ if $m\in\{0,1\}$, $\epsilon_m = 1$ if $m\geq 2$, and the summation range $M$ is over all solutions of $m_1+\cdots+m_d=k+m$ such that $1 \leq m_1,\ldots,m_{d-1}\leq K$ and $\max(m,1)\leq m_d\leq K$.
\end{corollary}

\begin{corollary}
The number of all vertical steps in the set of paths from $\mathcal P_{\mathcal A}(1,n)$ is equal to $|\mathcal F_{\mathcal A}(0,n)|$.
The number of all steps in the set of paths from $\mathcal P_{\mathcal A}(1,n)$ is equal to $|\mathcal F_{\mathcal A}(1,n)|$.
\end{corollary}

\begin{proposition}
If $N=1$ and $K=2$, then 
\[
P_1(x) = \frac{2(1-x)}{3x} - 
\frac{2\sqrt{\Delta}}{3x} \sin \left\{ 
	\frac{\pi }{6}+\frac{1}{3} \arccos \left(
	\frac{ 20x^3-6x^2+15x-2}{2 \Delta^{3/2}}
	\right)
\right\},
\]
where $\Delta = 1-5x-2x^2$.
\end{proposition}
\begin{proof}
By Corollary~\ref{cor:gener1}, for $N=1$ and $K=2$, we obtain three functional equations
\begin{align*}
	P_0(x) &= \frac{1}{1-x-x P_1(x)},
	\\
	P_1(x) &= 1 + x P_1(x) + x P_1(x)^2 + x P_2(x),
	\\
	P_2(x) &= \frac{1 - x P_1(x)}{1-x-x P_1(x)}
	= 1 + x P_0(x),
\end{align*}
which follow to the cubic equation
\[
	x^2 P_1(x)^3 + 2x(x-1) P_1(x)^2 + (1-x-2x^2) P_1(x) - 1 = 0.
\]
Using trigonometric methods we obtain the formula.
\end{proof}

\noindent \textbf{Remark}.
Let us give some first values of the sequences considered above for $N=1$ and $K=2$. That is, $\mathcal A = \{ V, U_1, U_0, D_1, D_2 \}$.
\begin{align*}
	(|\mathcal F_{\mathcal A}(0,n)|)_{n\geq 0} 
	&= (1, 3, 15, 84, 491, 2948, 18018, 111520, 696739, \ldots )
	\\
	(|\mathcal F_{\mathcal A}(1,n)|)_{n\geq 0} 
	&= (1, 6, 35, 207, 1251, 7678, 47658, 298371, 1880659, \ldots)
	\\
	(|\mathcal P_{\mathcal A}(0,n)|)_{n\geq 0} 
	&= (1, 2, 7, 30, 142, 716, 3771, 20502, 114194, 648276, \ldots)
	\\
	(|\mathcal P_{\mathcal A}(1,n)|)_{n\geq 0} 
	&= (1, 3, 10, 41, 190, 946, 4940, 26693, 147990, 837102, \ldots)
	\\
	(|\mathcal P_{\mathcal A}(2,n)|)_{n\geq 0}
	&= (1, 1, 2, 7, 30, 142, 716, 3771, 20502, 114194, 648276, \ldots)
\end{align*}

\subsection{{\L}ukasiewicz paths}

This section is devoted to the case where $K=1$ from the previous example. Namely, let $\mathcal B=\{V, U_N, U_{N-1},\ldots, U_0, D_1 \}$ for fixed $N\geq 0$. Let $\mathcal L = \mathcal B\setminus\{V\}$. Lattice paths consisting of steps from $\mathcal L$ are called $N$-{\L}ukasiewicz paths \cite{Roblet, Viennot1}. The weighted {\L}ukasiewicz paths encode several families of combinatorial structures like involutions, permutations, and set partitions, see Varvak \cite{Varvak}. Now we see that the proper weighted $N$-{\L}ukasiewicz paths encode $m$-primary $\mathcal B$-paths, where $m\in\{0,1\}$.

Namely, by Theorem~\ref{th:biject}, for $m\geq 0$ and $n\geq 1$, we have
\[
	|\mathcal P_{\mathcal B}(m,n)| = \sum_{\mu\in\mathcal P_{\mathcal L}(m,n)} w(\mu),
\]
if the weight function $w$ over steps from $\mathcal L$-paths is defined as follows
\begin{equation}
	\label{eq:wLukas}
	w_{}(D_1) = N+2, 
	\quad
	w_{}(U_{k}^{1,d}) 
	= \binom{N+2}{k+2},
	\quad
	w_{}(U_{k}^{0,d}) 
	= \binom{N+1}{k+1}.
\end{equation}
It is worth pointing out, that the weight function $w$ is independent of $d$.

\begin{example}
For $N=1$, the set $\mathcal W_{\mathcal L}(0,n)$ is the family of weighted Motzkin paths running from $(0,0)$ to $(n,0)$, which never go below the $x$-axis, and where the weight of the horizontal step $(1,0)$ is $2$ if it lies on the $x$-axis and $3$ if it lies above the $x$-axis, the weight $(1,1)$ is one and the weight of $(1,-1)$ is $3$.
\end{example}

\begin{corollary} 
Let $m\geq 0$ and $n\geq 1$, then
\begin{subequations} 
\begin{align}
	|\mathcal F_{\mathcal B}(m,n) |
	&=
	\sum_{k=0}^{\lfloor \frac{N n + m}{N+2} \rfloor}
	(-1)^k \binom{n}{k} \binom{(N+2)(n-k)+m}{2n},
	\\
	|\mathcal P_{\mathcal B}(0,n)| 
	&= (-1)^n + \sum_{j=1}^n \sum_{k=0}^{\lfloor \frac{N j + 1}{N+2} \rfloor}
	\frac{(-1)^{k+n-j}}{j} \binom{j}{k} \binom{(N+2)(j-k)}{2j-1},
	\\
	|\mathcal P_{\mathcal B}(1,n)| 
	&= \frac{1}{n}
	\sum_{k=0}^{\lfloor \frac{N n + 1}{N+2} \rfloor}
	 (-1)^k \binom{n}{k} \binom{(N+2)(n-k)}{2n-1}.
\end{align}
\end{subequations} 
\end{corollary}

\begin{corollary}
The number of all vertical steps in the set of paths from $\mathcal P_{\mathcal B}(1,n)$ is equal to $|\mathcal F_{\mathcal B}(0,n)|$.
The number of all steps in the set of paths from $\mathcal P_{\mathcal B}(1,n)$ is equal to $|\mathcal F_{\mathcal B}(1,n)|$.
\end{corollary}

\begin{corollary}
Let $P_{m}(x) = P_{\mathcal B,m}(x)$, then
\begin{align}
	P_{0}(x) 
	= 1 + x P_{0}(x) 
	\sum_{k=0}^N P_{1}(x)^k,
	\qquad
	P_{1}(x) 
	= 1 + x \sum_{k=0}^{N+1}
	P_{1}(x)^{k}.
\end{align}
\end{corollary}
\begin{proof}
Applying Proposition~\ref{prop:generV} for the weights given by \eqref{eq:wLukas} we obtain
\begin{align*}
	P_{0}(x) 
	&= 1 + x P_{0}(x) 
	\sum_{k=0}^N \binom{N+1}{k+1} (P_{1}(x) - 1)^k
	\\
	&= 1 + \frac{x P_{0}(x)}{(P_{1}(x) - 1)} 
	\sum_{k=1}^{N+1} \binom{N+1}{k} (P_{1}(x) - 1)^k.
\end{align*}
Using the binomial theorem and simplifying the result we obtain the functional equation for $P_0(x)$. In the same way one can show an analogous relation for $P_1(x)$.
\end{proof}

For instance, if $N=1$, then
\begin{align*}
	P_0(x) = \frac{1-x-\sqrt{1-6x-3x^2}}{2 x(1+x)},
	\quad
	P_1(x) = \frac{1-x-\sqrt{1-6x-3x^2}}{2x}.
\end{align*}
The generating functions of the sequences $(|\mathcal F_{\mathcal B}(m,n)|)_{n\geq 0}$, for $m\in\{0,1\}$, are derived by the author in \cite{mdRect} (see Eq. 18 for $m=0$ and Eq. 30a for $m=1$). That is,
\begin{align*}
	\sum_{n\geq 0} |\mathcal F_{\mathcal B}(0,n)| x^n 
	= \frac{1}{\sqrt{1-6x-3x^2}},
	\quad
	\sum_{n\geq 0} |\mathcal F_{\mathcal B}(1,n)| x^n 
	= \frac{1-x-\sqrt{1-6x-3x^2}}{2 x \sqrt{1-6x-3x^2}}.
\end{align*}

\noindent
\textbf{Remark}.
Let $N=K=1$.
\begin{align*}
	(|\mathcal F_{\mathcal B}(0,n)|)_{n\geq 0} 
	&= (1, 3, 15, 81, 459, 2673, 15849, 95175, 576963, \ldots) 
	&(A122868)
	\\
	(|\mathcal F_{\mathcal B}(1,n)|)_{n\geq 0} 
	&= (1, 6, 33, 189, 1107, 6588, 39663, 240894, \ldots) 
	&
	\\
	(|\mathcal P_{\mathcal B}(0	,n)|)_{n\geq 0} 
	&= (1, 2, 7, 29, 133, 650, 3319, 17498, 94525, \ldots)
	&(A064641)
	\\
	(|\mathcal P_{\mathcal B}(1,n)|)_{n\geq 0} 
	&= (1, 3, 9, 36, 162, 783, 3969, 20817, 112023, \ldots)
	&(A156016)
\end{align*}
The numbers starting with $A$ in parentheses denote corresponding sequences in OEIS \cite{oeis}.

\subsection{Infinite number of down steps}
In this section we consider the case where the set of steps contains infinitely many down steps. Namely, let $\mathcal C = \{ V, U_N, U_{N-1}, \ldots, U_0, D_1, D_2, \ldots \}$ for fixed $N\geq 0$.
Let $\mathcal L = \mathcal C \setminus \{V\}$. By Theorem~\ref{th:biject}, we have
\[
	|\mathcal P_{\mathcal C}(m,n)| = \sum_{\mu\in\mathcal P_{\mathcal L}(m,n)} w(\mu),
\]
if the weight function $w_{}$ over steps from $\mathcal L$-paths is defined as follows
\begin{align}
	w_{}(D_p) = \left\{ 
\begin{array}{cl}
	N+2 & \mathrm{if}\,\, p = 1, \\ 
	1 & \mathrm{if}\,\, p \geq 2,
\end{array}		
	\right.
	\qquad
	w_{}(U_k^{m,d}) = \binom{N-k+d+1-\epsilon_m}{N-k},
\end{align}
where $\epsilon_m = 0$ if $m\in\{0,1\}$ and $\epsilon_m = 1$ if $m\geq 2$.

\begin{corollary}
Let $m\geq 0$ and $n\geq 1$, then
\begin{subequations} 
\begin{align}
	|\mathcal F_{\mathcal C}(m,n) | 
	&= \binom{(N+2)n+m}{2n},
	\\
	|\mathcal P_{\mathcal C}(1,n)| 
	&= \frac{1}{n} \binom{(N+2)n}{2n-1}.
\end{align}
\end{subequations}
\end{corollary}
\begin{proof}
It follows from Proposition~\ref{prop:numfree} and Proposition~\ref{prop:numprim}. That is, 
\[
	|\mathcal F_{\mathcal C}(m,n) | 
	= [z^{N n + m}] \left( 1 - z \right)^{-2n-1},
	\qquad
	|\mathcal P_{\mathcal C}(1,n)| 
	= \frac{1}{n} [z^{N n+1}] \left( 1- z\right)^{-2n}.
\]
\end{proof}

\begin{corollary}
The expected number of vertical steps in a path from $\mathcal P_{\mathcal C}(1,n)$ is equal to $(Nn+1)/2$. The expected number of steps in a path from $\mathcal P_{\mathcal C}(1,n)$ is equal to $((N+2)n+1)/2$.
\end{corollary}

\noindent \textbf{Remark}. 
The array $(|\mathcal P_{\mathcal C}(1,n)|)_{N,n}$ for $0\leq N\leq 3$ and $0\leq n\leq 7$, is
\[
\left(
\begin{array}{cccccccc}
 1 & 2 & 2 & 2 & 2 & 2 & 2 & 2 \\
 1 & 3 & 10 & 42 & 198 & 1001 & 5304 & 29070 \\
 1 & 4 & 28 & 264 & 2860 & 33592 & 416024 & 5348880 \\
 1 & 5 & 60 & 1001 & 19380 & 408595 & 9104550 & 210905400
\end{array}
\right).
\]
The second row of the array is denoted by A007226 in OEIS \cite{oeis}.
The array $(|\mathcal F_{\mathcal C}(0,n)|)_{N,n}$ for $0\leq N\leq 3$ and $0\leq n\leq 7$, is
\[
\left(
\begin{array}{cccccccc}
 1 & 1 & 1 & 1 & 1 & 1 & 1 & 1 \\
 1 & 3 & 15 & 84 & 495 & 3003 & 18564 & 116280 \\
 1 & 6 & 70 & 924 & 12870 & 184756 & 2704156 & 40116600 \\
 1 & 10 & 210 & 5005 & 125970 & 3268760 & 86493225 & 2319959400
\end{array}
\right).
\]
The second row of the array is denoted by A005809 in OEIS \cite{oeis}.
The third one is A001448 in OEIS \cite{oeis}, etc.

\subsection{Dyck paths with vertical steps}

Originally, a \emph{Dyck path} \cite{Deutsch2} is a lattice path running from $(0,0)$ to $(2n,0)$ and consisting of steps $U_1$ and $D_1$, for $n\geq 0$.
In this section we consider generalized Dyck paths which contain additional vertical steps. Namely, let $\mathcal D=\{V, U_N, D_K \}$, for fixed $N,K\geq 1$. Let $\mathcal L=\{U_N, U_{N-1}, \ldots, U_0, D_1, D_K \}$. By Theorem~\ref{th:biject}, we have
\[
	|\mathcal P_{\mathcal D}(m,n)| = \sum_{\mu\in\mathcal P_{\mathcal L}(m,n)} w(\mu),
\]
if the weight function $w_{}$ over steps from $\mathcal L$-paths is defined as follows
\[
	w_{}(D_p) = \left\{ 
\begin{array}{cl}
	2 & \mathrm{if}\,\, p = 1, \\ 
	1 & \mathrm{if}\,\, p \geq 2.
\end{array}		
	\right.,
	\quad
	w_{}(U_{k}^{m,d}) = \binom{N-k+d-\epsilon_m}{N-k},
\]
where $\epsilon_m = 0$ if $m\in\{0,1\}$ and $\epsilon_m = 1$ if $m\geq 2$.

\begin{corollary}
Let $m\geq 0$ and $n\geq 1$, then
\label{cor:dyckcor1}
\begin{subequations} 
\begin{align}
	|\mathcal F_{\mathcal D}(m,n) |
	&= \sum_{k=0}^{\lfloor \frac{Nn + m}{N+K} \rfloor} 
	\binom{n}{k} \binom{n(N+1)-k(N+K)+m}{n},
	\\
	|\mathcal P_{\mathcal D}(1,n)|
	&= \frac{1}{n} \sum_{k=0}^{\lfloor \frac{Nn + 1}{N+K} \rfloor} 
	\binom{n}{k} \binom{n(N+1)-k(N+K)}{n-1}.
\end{align}
\end{subequations}
\end{corollary}
\begin{proof}
It follows from Proposition~\ref{prop:numfree} and Proposition~\ref{prop:numprim}. That is,
\begin{align*}
	|\mathcal F_{\mathcal S}(m,n) | 
	= [z^{N n + m}] \frac{\Big( 1 + z^{N+K} \Big)^n}{( 1 - z )^{n+1}},
	\quad
	|\mathcal P_{\mathcal S}(1,n)| 
	= \frac{1}{n} [z^{Nn+1}] 
	\frac{\Big( 1 + z^{N+K} \Big)^n}{( 1 - z )^{n}}.
\end{align*} 
Using \eqref{eq:expans1} we obtain the required formulas.
\end{proof}

\begin{corollary}
If $K=1$ then for $n\geq 1$ we have
\begin{align*}
	|\mathcal P_{\mathcal S}(0,n)| = (-1)^n + \sum_{j=1}^n \sum_{k=0}^{\lfloor \frac{Nj + 1}{N+1} \rfloor}   \frac{(-1)^{n-j}}{j} 
	\binom{j}{k} \binom{(N+1)(j-k)}{j-1}.
\end{align*}
\end{corollary}
\begin{proof}
It follows from Proposition~\ref{prop:P0} and Corollary~\ref{cor:dyckcor1}.
\end{proof}

\begin{corollary}
Let $P_m(x) = P_{\mathcal D,m}(x)$, then
\begin{equation}
\begin{split}
	P_{0}(x) 
	&= 1 + \delta_0 x P_{0}(x)
	+ x P_{0}(x) \sum_{k=1}^N
	\sum_{d=1}^{k} 
	\binom{N-k+d}{d}	
	\sum_{M}
	\prod_{j=1}^d 
	(P_{m_j}(x)-1),
	\\
	P_{m}(x)
	&= 1 + \delta_m x + x \sum_{k=0}^N \sum_{d=1}^{k+1}  
	\binom{N-k+d-\epsilon_m}{N-k}   
	\sum_{M}
	\prod_{j=1}^d (P_{m_j}(x)-1),
\end{split}
\end{equation}
where $\delta_m = |\mathcal D_{-m}|$ if $m\in\{0,1\}$, $\delta_m = 1$ if $D_m\in\mathcal D$, and $\delta_m = 0$ if $D_m\notin\mathcal D$, for $m\geq 2$.
Further, the summation range $M$ is over all solutions of $m_1+\cdots+m_d=k+m$ such that $1 \leq m_1,\ldots,m_{d-1}\leq K$ and $\max(m,1)\leq m_d\leq K$.
\end{corollary}

For instance, if $K=N=1$, then
\[
	P_0(x) = \frac{1-\sqrt{1-4x -4x^2}}{2 x(1+x)},
	\qquad
	P_1(x) = \frac{1 - \sqrt{1-4x-4x^2}}{2x}.
\]

\vspace{0.2cm}
\noindent \textbf{Remark}. Let $N=K=1$.
\begin{align*}
	(|\mathcal F_{\mathcal D}(0,n)|)_{n\geq 0}
	&= (1, 2, 8, 32, 136, 592, 2624, 11776, 53344, 243392, \ldots)
	&(A006139)
	\\
	(|\mathcal F_{\mathcal D}(1,n)|)_{n\geq 0}
	&= (1, 4, 16, 68, 296, 1312, 5888, 26672, 121696, \ldots)
	&(A179191)
	\\
	(|\mathcal P_{\mathcal D}(0,n)|)_{n\geq 0}
	&= (1, 1, 3, 9, 31, 113, 431, 1697, 6847, 28161, 117631, \ldots) 
	&(A052709)
	\\
	(|\mathcal P_{\mathcal D}(1,n)|)_{n\geq 0}
	&= (1, 2, 4, 12, 40, 144, 544, 2128, 8544, 35008, 145792, \ldots)
	&(A025227)
\end{align*}
The numbers starting with $A$ in parentheses denote corresponding sequences in OEIS \cite{oeis}.

\subsection{Delannoy paths}
\label{sec:delan}

A \emph{Delannoy path} is a lattice path from $(0,0)$ to $(n,k)$ in $\mathbb Z\times\mathbb Z$ consisting of steps $(1,0)$, $(1,0)$, and $(0,1)$. The number of Delannoy paths running from $(0,0)$ to $(n,k)$ is called \emph{Delannoy number} \cite{Delannoy} and denoted by $D(n,k)$. The number of these paths running from $(0,0)$ to $(n,n)$ and never go below the line $y=x$ is called \emph{central Delannoy number} \cite{Fray, Hetyei} and denoted by $D(n)$. It is well-known that
\[
	D(n,k) = \sum_{j=0}^{k}  \binom{n}{j} \binom{n+k-j}{n}.
\]
These numbers are denoted in OEIS \cite{oeis} by A152250 and A001850.
Additionaly, let us denote by $S(n)$ the number of central Delannoy paths running from $(0,0)$ to $(n,n)$ that do not go below the line $y=x$. The numbers of such paths are called the large Schr{\"o}eder numbers \cite{Deutsch1} and they are denoted by A006318 in OEIS \cite{oeis}.

\begin{proposition}
\label{prop:delannoy}
Let $\mathcal E = \{V, U_1, U_0\}$ and $m,n\geq 0$. Then 
\begin{equation}
	D(n,m) = |\mathcal F_{\mathcal E}(m-n,n)|,
	\qquad
	S(n) = |\mathcal P_{\mathcal E}(1,n)| = |\mathcal P_{\mathcal E}(0,n)|.
\end{equation}
\end{proposition}

\begin{proof}
We obtain required bijection by transforming lattice points by the rule $(i,j) \mapsto (i-j,i)$ together with preserving connections between lattice points. Indeed, step $V$ becomes $(1,0)$, $U_1$ becomes $(0,1)$, and $U_0$ becomes $(1,1)$. Additionally, we remove the last vertical step in every path from $\mathcal P_{\mathcal E}(1,n)$.
\end{proof}

Let $\mathcal L = \{U_1, U_0, D_1\}$. By Theorem~\ref{th:biject}, we have
\[
	S(n) = \sum_{\mu\in\mathcal P_{\mathcal L}(1,n)} w(\mu),
\]
if the weight function $w$ is defined as 
\[
	w(U_1) = 1,
	\qquad
	w(U_0) = 3,
	\qquad
	w(D_1) = 2.
\]

\begin{corollary}
The expected number of steps $(1,1)$ in a central Delannoy path running from $(0,0)$ to $(n,n)$ which never goes below the line $y=x$ is
\[
	n \frac{D(n,n-1)}{D(n,n+1) - D(n,n)} = 
	n \left(\sum_{j=0}^{n-1}  \binom{n}{j} \binom{2n-j-1}{n} \right)
	\left(
	\sum_{j=0}^{n}  \binom{n}{j} \binom{2n-j}{n-1}\right)^{-1}.
\]

\end{corollary}
\begin{proof}
By Proposition~\ref{prop:delannoy}, the expected number is the total number of steps $U_0$ in paths from $\mathcal P_{\mathcal V}(1,n)$ divided by the size of $\mathcal P_{\mathcal V}(1,n)$. By Theorem~\ref{th:prop1}, we have
$\#Steps(U_0 \in \mathcal P_{\mathcal V}(1,n)) = |\mathcal F_{\mathcal V}(1, n-1)|$. That is, required number is $D(n,n-1)/S(n)$. By \eqref{eq:Pv1n}, $S(n) = (D(n,n+1) - D(n,n))/n)$.
\end{proof}

\bibliographystyle{plain}

\end{document}